\documentclass[12pt]{article}

\usepackage[dvips]{graphicx}

\usepackage[bookmarks,colorlinks]{hyperref}

\usepackage{amsmath,amstext,amssymb,amsopn,amsthm}
\usepackage{amsmath,amssymb,amsthm}
\usepackage[mathscr]{eucal}

\usepackage{graphicx}
\usepackage{amssymb}
\usepackage{amsmath}
\usepackage{color}
\usepackage{times}

\newcommand{\sphere}{\mathbb{S}}
\newcommand{\F}{\mathcal{F}}

\newcommand{\R}{\mathbb{R}}
\newcommand{\N}{\mathbb{N}}
\newcommand{\E}{\mathbb{E}}

\newcommand{\dY}{\Delta Y}

\newcommand{\e}{\varepsilon}

\newcommand{\LL}{{L}}

\newcommand{\Lt}{{\LL^2}}
\newcommand{\Rd}{{\R^d}}

\newcommand{\Rn}{{\R^n}}

\newcommand{\C}{\mathbb{C}}

\newcommand{\M}{{\mathcal{M}}}

\newcommand{\I}{{\mathbb{I}}}
\newcommand{\Lp}{{\LL^p}}

\newcommand{\Ao}{{A}}
\newcommand{\At}{{B}}

\newcommand{\mB}{{W}}
\newcommand{\wt}{\widetilde}

\newtheorem{theorem}{Theorem}

\DeclareMathOperator{\sgn}{sgn}

\begin{document}
\thispagestyle{plain}
\sloppy

\title{Parabolic martingales and non-symmetric Fourier multipliers}
\author{Krzysztof Bogdan\thanks{Corresponding author: Institute of Mathematics of the
Polish Academy of Sciences, and
Institute of Mathematics,
 Wroc{\l}aw University of Technology,
50-370 Wroc{\l}aw, Poland, bogdan@im.pwr.wroc.pl. Supported in part by grant MNiSW N N201 397137.}\and  {\L}ukasz Wojciechowski\thanks{Mathematical Institute, University of Wroc{\l}aw, 50-384 Wroc{\l}aw, Poland, luwoj@math.uni.wroc.pl. }
} \maketitle

\begin{abstract}
We give a class of Fourier multipliers with non-symmetric symbols and explicit norm bounds on $L^p$ spaces by using the stochastic calculus of L\'evy processes and Burkholder-Wang estimates for differentially subordinate martingales. 
\end{abstract}
\footnotetext{2010 {\it MS Classification}: 42B15, 60G15, 60G46.}
\section{Introduction and main result}
For each function $m:\,\Rd\to \C$ of absolute value bounded by $1$, there is a unique 
linear contraction $M$ on $\Lt(\Rd) $ defined in terms of the Fourier
transform by
\begin{equation}
  \label{eq:dmf}
\widehat{M f}=m\hat f\,,
\end{equation}
or, in terms of bilinear forms and Plancherel theorem, by
\begin{equation}
   \label{eq:dmf2}
\Lambda(f,g)=\int_{\Rd}{M f(x) g(x) dx}= (2 \pi)^{-d}\int_{\Rd} m(\xi) \widehat{f}(\xi)\widehat{g}(-\xi) d\xi\,.
\end{equation}
We are interested in {\it symbols} $m$ for which the {\it Fourier multiplier} $M$ has a finite operator norm $\|M\|_p$
on $L^p(\Rd)$ for all $p\in (1,\infty)$:
\begin{equation}\label{szac:mn}
|\Lambda (f,g)|\leq \|M\|_p\|f\|_p\|g\|_q,
\end{equation}
where $q=p/(p-1)$ and, say, $f,g \in C_c^{\infty}(\Rd).$
Motivated by \cite{banuelos-mendez, NV}, a wide class of multipliers was recently studied in \cite{MR2345912, 2011RBKBAB} by transforming the so-called parabolic martingales of L\'evy process. Burkholder-Wang inequalities for differentially subordinate
 martingales (\cite{MR1334160}) were used to bound their norms:
\begin{equation}
  \label{eq:psm1}
\|M\|_p\le \max\{p-1, \frac{1}{p-1}\}=:p^*-1.
\end{equation}
Surprisingly, the symbols $m$ obtained in \cite{MR2345912, 2011RBKBAB} turned out to be symmetric, even when non-symmetric L\'evy processes were used in the construction. In this paper we propose a new approach which leads to {\it non-symmetric} symbols. Namely
we use two different
L\'evy processes to drive the martingales defining the pairing $\Lambda$. Compared to  \cite{MR2345912, 2011RBKBAB} we also slightly modify the calculations of the Fourier symbol.

Let $d,n\in \N$ and consider the general L\'evy-Khinchine exponent on $\R^n$,
\begin{eqnarray}\label{eq:l-c}
\Psi(\zeta) = \int_{\R^n} \left(e^{i(\zeta,z)} -1 - i(\zeta,z) 1_{|z| \leq 1}\right)\nu(dz) - \frac{1}{2}\int_{\sphere}\left( \zeta, \theta \right)^2 \mu(d\theta)+ i(\zeta, \gamma),
\end{eqnarray}
where $\zeta,\gamma \in \R^n$, $\mu\geq 0$ is a (non-unique) finite 
measure on the unit sphere $\sphere\subset \R^n$, and $\nu\geq 0$ is a (unique) L\'evy measure on $\Rn$: 
$ \nu(\{0\})=0$ and
\begin{equation}
  \label{eq:clm}
\int_\Rn \min(|z|^2,1) \nu(dz)<\infty\nonumber.
\end{equation}
Here $(\xi,\eta)=\sum_{k} \xi_k\eta_k$ and  $|\xi|^2=\sum_{k} |\xi_k|^2=(\xi,\overline{\xi})$ for $\xi,\eta \in \R^d$, $\R^n$, $\C^d$, $\C^n$. 
Consider 
complex-valued functions $\phi$ on
$\Rn$ and $\varphi$ on $\sphere$ such that $\|\phi\|_\infty\leq 1$ and $\|\varphi\|_\infty\leq 1$. For $\zeta\in \R^n$ we let
\begin{eqnarray}\label{eq:l-ct}
\wt\Psi(\zeta) = \int_{\R^n} \left(e^{i(\zeta,z)} -1 - i(\zeta,z) 1_{|z| \leq 1}\right)\phi(z)\nu(dz) - \frac{1}{2}\int_{\sphere}\left( \zeta, \theta \right)^2 \varphi(\theta)\mu(d\theta).
\end{eqnarray}
Let $A ,B \in \R^{d \times n}$. For $\xi\in \Rd$ we define
\begin{align} \label{mn:gen3}
&m(\xi)=\bigg[e^{\Psi(B^T\xi - A^T\xi)} - e^{\Psi(\At ^T \xi) +\Psi(-\Ao ^T \xi)} \bigg]\times\\
&
\frac{
        \int\limits_\Rd \left(e^{i(B^T\xi,z)}-1 \right)\left(e^{i(-A^T\xi,z)}-1 \right)\phi\left(z\right)\nu(dz) 
        -
             \int\limits_{\sphere} \left( B^T\xi , \theta \right)\left( -A^T\xi ,\theta \right) \varphi\left( \theta \right) \mu\left(d\theta \right)
          }
     {
    \int\limits_\Rd \left(e^{i(B^T\xi,z)}-1 \right)\left(e^{i(-A^T\xi,z)}-1 \right)\nu(dz) 
        -
             \int\limits_{\sphere} \left( B^T\xi , \theta \right) \left( -A^T\xi ,  \theta \right) \mu\left(d\theta \right)
     }  \nonumber ,
\end{align} 
with the convention that
\begin{align} \label{mn:gen4}
&m(\xi)=\e^{\Psi(\At ^T \xi) +\Psi(-\Ao ^T \xi)}\times\\
& \big(\int\limits_\Rd \!\left(e^{i(B^T\xi,z)}-1 \right)\left(e^{i(-A^T\xi,z)}-1\right)\phi\left(z\right)\nu(dz) 
        \!-\!
\int\limits_{\sphere} \! \left( B^T\xi , \theta \right)\left( -A^T\xi ,\theta \right) \varphi\left( \theta \right) \mu(d\theta)
\big)\nonumber,
\end{align} 
if the denominator in \eqref{mn:gen3} is zero. 
To simplify \eqref{mn:gen3} and \eqref{mn:gen4}, we note that
\begin{align}
&\int_{\R^n} 
\left(e^{i(\zeta_1,z)} -1\right)\left(e^{i(\zeta_2,z)} - 1\right)\phi(z)\nu(dz) 
- \int_{\sphere}\left( \zeta_1, \theta \right)\left(\zeta_2,\theta\right) \varphi(\theta)\mu(d\theta)\nonumber\\
&=\wt\Psi(\zeta_1+\zeta_2) -
\wt\Psi(\zeta_1)-\wt\Psi(\zeta_2), \quad \zeta_1, \zeta_2\in \R^n, \label{eq:l-ctu}
\end{align}
and a similar identity holds for the special case of $\Psi$. Thus, $m(\xi)$ equals
\begin{align}\label{mn:genca}
\left[\!e^{\Psi(B^T\xi-A^T\xi)}\!-\!e^{\Psi(B^T\xi)+\Psi(-A^T\xi)}\!\right]
\frac{\!\wt \Psi(B^T\xi- A^T\xi)\!-\!\wt \Psi(B^T\xi)\!-\!\wt \Psi(-A^T\xi)\!}
{\!\Psi(B^T\xi- A^T\xi)\!-\!\Psi(B^T\xi)\!-\!\Psi(-A^T\xi)\!},
\end{align}
with the convention that
\begin{align}\label{mn:genca0}
m(\xi)&=\!e^{\Psi(B^T\xi)+\Psi(-A^T\xi)}
\left[{\!\wt \Psi(B^T\xi- A^T\xi)\!-\!\wt \Psi(B^T\xi)\!-\!\wt \Psi(-A^T\xi)\!}\right],
\end{align}
if the denominator in \eqref{mn:genca} is zero.
In short,
\begin{align}\label{mn:genc}
m(\xi)=&e^{\Psi(B^T\xi)+\Psi(-A^T\xi)}\;\left[{\wt \Psi(B^T\xi- A^T\xi)-\wt \Psi(B^T\xi)-\wt \Psi(-A^T\xi)}\right]\\
&\times q\left(\Psi(B^T\xi -A^T\xi)-\Psi(B^T\xi)-\Psi(-A^T\xi)\right),
\nonumber
\end{align}
where
$$q(z)=(e^z-1)/z\;\quad \mbox{ if $z\in \C\setminus \{0\}$, } \qquad \mbox{ and }\quad q(0)=1\,.$$
We see that (\ref{mn:gen3}, \ref{mn:gen4}) are equivalent to \eqref{mn:genc}.
Here is our main result.
\begin{theorem}\label{th:bfm}
If $\M$ satisfies
{\rm (\ref{eq:dmf})} and {\rm (\ref{mn:genc})}, and $1<p<\infty$, then $\|\M \|_p\leq p^*-1$.
\end{theorem}

Theorem~\ref{th:bfm} is proved in Section~\ref{sec:j} by using stochastic calculus of L\'evy processes. In Section~\ref{sec:e} we make some clarifying comments 
and point out a few symbols resulting from \eqref{mn:genc}.
An alternative approach for Gaussian L\'evy processes is given in Section~\ref{sec:c}, where we use the  familiar and more compact classical It\^o calculus.
This, however, boils down to taking $\nu=0$ in \eqref{eq:l-c},
and yields only symmetric symbols.
Details of  the stochastic calculus needed in this note may be found in \cite{MR2345912, 2011RBKBAB}. We refer to \cite{MR1406564, MR1739520} for information on L\'evy processes, including compound Poisson processes, and to \cite{MR745449,MR688144, MR2020294} for various expositions of stochastic calculus. 
Burkholder's method is discussed in depth in \cite{2010RB},  and a classical treatment of Fourier multipliers may be found in \cite{MR0290095}. A recent study of non-symmetric homogeneous symbols is given in \cite{2010arXiv1008.3044M}. 
As we already remarked, multipliers with symmetric symbols were obtained by similar methods in \cite{MR1406564, MR1739520}, and they include, e.g., Marcinkiewicz-type fractional multipliers, the Beurling-Ahlfors operator and the second order Riesz transforms. 
We also note that  the bound \eqref{eq:psm1} cannot in general be improved, because it is optimal for second order Riesz transforms (\cite{MR1739520,2011arXiv1111.7212B}). 

While we considerably extend the class of symbols manageable by our methods, we fall short of non-symmetric symbols homogeneous of degree $0$.  Specifically, homogeneous symbols may appear as the second factor (the ratio) in \eqref{mn:gen3} or \eqref{mn:genca}, but they are tempered at the origin and infinity by the first factor therein, which involves the Fourier transform of the semigroup. 
Replacing $\Psi$ and $\wt\Psi$ by $u\Psi$ and $u\wt\Psi$  and letting $u\to \infty$ usually removes the first factor in \eqref{mn:gen3} and \eqref{mn:genca} if $A=B$. The resulting symbols are given in \eqref{mn:gencd} below, and include many symmetric symbols homogeneous of degree $0$, see \eqref{eq:Bcs}.
We wonder if a different pairing or other modifications of our methods could produce symbols which are both discontinuous and non-symmetric.

Below we will often use the quadratic variation $[F,F]$ and covariation $[F,G]$ of square-integrable continuous-time c\`adl\`ag martingales $F$, $G$. Recall that $[F,F]$ is the unique adapted right-continuous non-decreasing process with jumps $[F,F]_t-[F,F]_{t-}=(F_t-F_{t-})^2$, and such that $t\mapsto F^2_t - [F,F]_t$ is a (continuous) martingale starting at $0$ (\cite[VII.42]{MR745449}). 
We say that $F$ is {\it differentially subordinate} to $G$ if $t\mapsto [G,G]_t-[F,F]_t$ is nonnegative and non-decreasing (\cite{MR1334160}). The covariation $[F,G]$ is defined by polarization, and we have $\E F_t G_t=\E[F,\overline{G}]_t$.
All the functions and measures considered in this paper are assumed to be Borelian.

\section{Proof of Theorem~\ref{th:bfm}}\label{sec:j}

We will first prove the result for 
\begin{equation}\label{eq:tw1}
\Psi(\zeta)= \int_{\Rd} \left(e^{i(\zeta,z)}-1\right)\nu(dz) \,, \qquad \zeta \in \Rn,
\end{equation}  
and
\begin{eqnarray}\label{eq:l-ct1}
\wt\Psi(\zeta) = \int_{\R^n} \left(e^{i(\zeta,z)} -1\right)\phi(z)\nu(dz)
 \,, \qquad \zeta \in \Rn,
\end{eqnarray}
where $\nu$ is finite. To this end 
we only need to define $\Lambda$ satisfying (\ref{eq:dmf2}) and (\ref{szac:mn}).

By $f$ and $g$ below we will denote complex-valued smooth compactly supported (i.e. $C^\infty_c$) functions on $\Rd$ or $\R^n$. 
Let $(Y_t,t\ge 0)$ be a compound Poisson process on $\R^n$ with the L\'evy measure $\nu$, semigroup $(P_t)$, expectation $\E$ and jumps $\dY_t=Y_t-Y_{t^-}$. Let $x\in \R^n$. Recall that $P_tf(x)=\E f(x+Y_t)=\int_\Rd f(x+y)p_t(dy)$, where $t\geq 0$,
$$p_t = e^{-t|\nu|} \sum_{n=0}^{\infty}\frac{\nu^{*n}}{n!},$$
and $\hat{p}_t(\zeta)=\E e^{i(\zeta,Y_t)}=e^{t\Psi(\zeta)}$ for $\zeta\in \R^n$.
The process $(\Ao Y_t, t\geq 0)$ is compound Poisson, too, with 
the L\'evy measure equal to (the pushforward measure) $\Ao\nu=\nu\circ A^{-1}$
on $\Rd\setminus\{0\}$  (\cite[Proposition~11.10]{MR1739520}). Indeed, for $\xi\in \Rd$,
$$\E e^{i(\xi,\Ao Y_t)}=e^{t
\Psi(\Ao^T\xi)}
=
\int_{\Rn} \left(e^{i(\xi,\Ao z)}-1\right)\nu(dz)
=\int_{\Rd} \left(e^{i(\xi,z)}-1\right)\Ao\nu(dz).
$$
We also have
$\E f(x+\Ao Y_t)=\int f(x+\Ao y)p_t(dy)=P_t^{\Ao} f(x)$,
where $$P^\Ao_t f(x)=\int f(x+\Ao y)p_t(dy).$$
We proceed similarly for $(\At Y_t, t\geq 0)$.
We remark that $(\Ao Y_t)$ and $(\At Y_t)$  have fairly general dependence structure, e.g. yield pairs of projections of $Y$.

We consider the filtration $\F_t = \sigma \{ Y_s: 0\leq s \leq t\}$. 
For $0\le t\le 1$ we define the {\it parabolic} martingale 
$F_t= F_t(x;f,\Ao )$, where
\begin{eqnarray*}
F_t(x;f,\Ao  ) &=& \E [f(x+\Ao  Y_1 )|\F_t] 
= \E [f(x+\Ao  (Y_1-Y_t) + \Ao Y_t  )|\F_t]\\
&=&  \int_{\Rd} f(x+\Ao y+\Ao Y_t )p_{1-t}(dy) = P_{1-t}^{\Ao }f(x+\Ao Y_t).
\end{eqnarray*}
Thus $F$ is of {\it function-type}, i.e. a composition of a (parabolic) function with a (space-time) stochastic process.
\noindent By It\^o formula \cite[p.17]{2011RBKBAB} for  $(\Ao Y_t)$,
\begin{eqnarray*}
&&F_t - F_0= \sum_{\substack{0 < v \leq t \\ \dY_v \neq 0}}
[P^{\Ao }_{1-v}f(x+\Ao  Y_v) - P^{\Ao }_{1-v}f(x+\Ao  Y_{v-})]\\
 &&- \int_0^t \int_{\Rd} [P^{\Ao }_{1-v}f(x+\Ao  (Y_v +z)) - P^{\Ao }_{1-v}f(x+\Ao  Y_v)]\nu(dz)dv.
\end{eqnarray*}
Following \cite{MR2345912, 2011RBKBAB} we also define more {\it general} (i.e. non function-type) martingales
\begin{eqnarray*}
&&G_t(x;g,\At  ,\phi) = \sum_{\substack{0 < v \leq t \\ \dY_v \neq 0}}[P^{\At }_{1-v}g(x+\At  Y_v) - P^{\At }_{1-v}g(x+\At  Y_{v-})]\phi(\dY_v)\\
 &&- \int_0^t \int_{\Rd} [P^{\At }_{1-v}g(x+\At  (Y_v +z)) - P^{\At }_{1-v}g(x+\At  Y_v)]\phi(z)\nu(dz)dv
\end{eqnarray*}
driven by $(\At Y_t)$. We see that
$F_t(x;f,\At ) = G_t(x;f,\At ,1)$. 
Let
\begin{equation}
\Lambda(f,g)= \int_{\Rd} \E F_1(x;f,\Ao  )G_1(x;g,\At  ,\phi)dx.
\end{equation}
\noindent By \cite[p.17]{2011RBKBAB}, $G_t:=G_t(x;g,\At  ,\phi)$ has quadratic variation
$$[G,G]_t = \sum_{0<v\leq t}| P^{\At }_{1-v}g(x+\At Y_v) -  P^{\At }_{1-v}g(x+\At Y_{v-})|^2|\phi(\dY_v)|^2.$$
The quadratic variation of $F$ is
$$[F,F]_t = |F_0|^2 + \sum_{0<v\leq t}| P^{\Ao }_{1-v}f(x+\Ao Y_v) -  P^{\Ao }_{1-v}f(x+\Ao Y_{v-})|^2.$$
Thus, $G(x;g,\At,\phi)$ is differentially subordinate to $F(x;g,\At )$.
 Let $p,q \in (1, \infty)$ and $1/p + 1/q =1.$ By Fubini-Tonelli,
\begin{eqnarray}\label{eq777}
&&\int_{\R^d}\E|F_1(x;f,\Ao )|^p dx = \int_{\R^d} \E|f(x+\Ao  Y_1)|^p dx
= 
\int_{\R^d}\int_{\R^d} |f(x+\Ao  y)|^p p_1(dy)dx\nonumber\\
&&=\int_{\R^d}\int_{\R^d} |f(x)|^p p_1(dy) dx =||f||_p^p.
\end{eqnarray}
We then use Burkholder-Wang theory (\cite{MR1334160}) and the identity $p^*-1=q^*-1$:
$$\E|G_1|^q \leq (q^*-1)^q \E|g(x+\At Y_1)|^q 
= (p^*-1)^q \E|g(x+\At Y_1)|^q
.$$
Following (\ref{eq777}), we now obtain 
$$\int_{\Rd}\E|G_1(x;g,\At ,\phi)|^q dx \leq (p^*-1)^q \int_{\Rd} |g(x)|^q dx.$$
By H\"older inequality,
$|\Lambda(f,g)|\leq 
(p^* -1) ||f||_p ||g||_q$, 
as required in  \eqref{szac:mn}.  
To obtain \eqref{eq:dmf2}, we recall that $\E F_1 G_1 = \E [F, \overline{G}]_1.$ Furthermore,
\begin{equation}\label{eq:trPA}
\widehat{P_t^{A} f}(\xi) = \widehat{f}(\xi) e^{t\Psi(-A^T\xi)}.\nonumber
\end{equation}
By this, the L\'evy system (see \cite{2011RBKBAB, MR1334160}) and Plancherel theorem,
\begin{eqnarray*}
&&\Lambda(f,g)= \int_{\Rd} \E \sum_{\substack{0< v \leq 1\\\dY_v \neq 0}}[P^{\Ao }_{1-v}f(x+\Ao Y_v) - P^{\Ao }_{1-v}f(x+\Ao  Y_{v-})]\\
&&\qquad\qquad\qquad\qquad\times [P^{\At }_{1-v}g(x+\At Y_v) - P^{\At }_{1-v}g(x+\At  Y_{v-})]\phi(\dY_v)dx\\
&=& \int_{\Rd}  \int_0^1\int_\Rd  \int_\Rd [P^{\Ao }_{1-v}f(x+\Ao (y+z)) - P^{\Ao }_{1-v}f(x+\Ao  y)]\\
&&\qquad\qquad\qquad\times[P^{\At }_{1-v}g(x+\At (y+z)) - P^{\At }_{1-v}g(x+\At  y)]\phi(z)\nu(dz)p_v(dy)dvdx\\
&=& (2\pi)^{-d}\int_{\Rd}m(\xi)\widehat{f}(\xi)\widehat{g}(-\xi)d\xi,
\end{eqnarray*}
where
\begin{eqnarray}
\nonumber
m(\xi)&=&  \int_0^1 \int_{\Rd} \int_{\Rd}\bigg( e^{-i(\xi,\Ao (y+z))} - e^{-i(\xi,\Ao  y )}\bigg)\bigg( e^{i(\xi,\At (y+z) )} - e^{i(\xi,\At  y )}\bigg)\\\nonumber
&& \qquad\qquad\qquad\times e^{(1-v)\Psi(-\Ao ^T\xi)}  e^{(1-v)\Psi(\At ^T\xi)}\phi(z)\nu(dz)p_v(dy)dv  \\\nonumber
&=& \int_0^1  \int_{\Rd} \int_{\Rd} e^{i(B^T\xi - A^T\xi, y)} e^{(1-v)(\Psi(\At ^T\xi)+\Psi(-\Ao ^T\xi))} \\\nonumber
&&\qquad\qquad \times \bigg( e^{i(\xi,\At  z)} - 1\bigg) \bigg( e^{-i(\xi,\Ao  z)} - 1\bigg)\phi(z)\nu(dz)p_v(dy)dv\\\label{eq:wnmpc}
&=& \int_0^1  \int_{\Rd} e^{v\Psi(B^T\xi - A^T\xi)}e^{(1-v)(\Psi(\At ^T\xi)+\Psi(-\Ao ^T\xi))} \\\nonumber
&&\qquad   \qquad\qquad\times \bigg( e^{i(\xi,\At  z)} - 1\bigg) \bigg( e^{-i(\xi,\Ao  z)} - 1\bigg)\phi(z)  \nu(dz)dv.
\end{eqnarray} 
We directly verify (compare \eqref{eq:l-ctu}) that
$$\int\limits_{\Rd}\bigg( e^{i(\xi,\At  z)} - 1\bigg)\bigg( e^{-i(\xi,\Ao z)} - 1\bigg)\phi(z) \nu(dz)=\wt\Psi(B^T\xi - A^T\xi)-\wt\Psi(\At ^T\xi)-\wt\Psi(-\Ao ^T\xi).$$ 
We integrate \eqref{eq:wnmpc} with respect to $dv$ and obtain \eqref{mn:genc}.

We shall next give an extension to compound Poisson processes with drift. 
We claim that the multiplier resulting from $\phi$ and the L\'evy - Khinchine exponent
\begin{eqnarray}
\nonumber
 \int_{\Rd}\left(e^{i(\xi,z)}-1-i(\xi,z)1_{|z|\le 1}\right)\nu(dz)  +i(\xi,\gamma)
=\int_{\Rd}(e^{i(\xi,z)}-1)\nu(dz)  +i(\xi,h),
\label{eq:dd}
\end{eqnarray}
where $h=\gamma-\int_{\Rd}z1_{|z|\le 1}\nu(dz)$, 
has the norm bounded by $p^*-1$ on $L^p(\Rd)$, too.
The operator $T_h f(x)=f(x-h)$ is an isometry of $L^p(\Rd)$, and also a Fourier multiplier with symbol $e^{i(\xi,h)}$. We can multiply $m(\xi)$ in \eqref{mn:genc}
by
$e^{i(B^T\xi - A^T\xi,h)}$, without changing the norm of the multiplier. 
The exponential function absorbs into the first factor on the right-hand side of \eqref{mn:genc}, which grants the extension.

We will now pass to general L\'evy processes, i.e. arbitrary $\Psi$ and $\wt\Psi$ given by \eqref{eq:l-c} and \eqref{eq:l-ct}.
We first note that the norm bound of our multipliers is preserved under pointwise convergence of the symbols, which follows from Plancherel theorem and Fatou's lemma in the same way as in \cite[the proof of Theorem 1.1]{2011RBKBAB}.
Then we remark that $m$ in \eqref{mn:genc} depends continuously on $\Psi$ and $\wt\Psi$.
Finally we recall the following approximation procedure:
let $\varepsilon\to 0^+$,
$$\nu_\varepsilon= 1_{\{ |z| > \varepsilon \}}\nu \ ,\quad  \mbox{ and }\quad 
  \mu_{\varepsilon}(drd\theta) = \varepsilon^{-2}\delta_{\varepsilon}(dr)\mu(d\theta) \,.
$$
Here $(r,\theta) \in (0,\infty)\times\sphere$ are the polar coordinates in $\Rn$
and $\delta_\varepsilon$ is the probability measure concentrated at $\varepsilon$.
We consider
\begin{eqnarray*}
\nonumber
\Psi_\varepsilon(\xi) &=& \int_{\Rd}\left(e^{i(\xi,z)}-1-i(\xi,z)1_{|z|\le 1}\right)(\nu_\varepsilon+\mu_\varepsilon)(dz)  +i(\xi,\gamma),
\end{eqnarray*}
and
\begin{eqnarray*}
\nonumber
\wt\Psi_\varepsilon(\xi) &=& \int_{\Rd}\left(e^{i(\xi,z)}-1-i(\xi,z)1_{|z|\le 1}\right)\phi_\varepsilon(z)(\nu_\varepsilon+\mu_\varepsilon)(dz),
\end{eqnarray*}
where
$\phi_\varepsilon(z)=1_{\{ |z| > \varepsilon \}}\phi(z) + 1_{\{ |z| = \varepsilon \}}\varphi(z/|z|)$.
By dominated convergence, $\Psi_\varepsilon(\zeta)\to \Psi(\zeta)$ and 
$\wt\Psi_\varepsilon(\zeta)\to \wt\Psi(\zeta)$ (see \cite[(3.3)]{2011RBKBAB}), which yields the convergence of the resulting symbols (say, $m_\varepsilon$) to $m$ in \eqref{mn:genc}, and ends the proof.
\qed

\section{Comments and examples}\label{sec:e}
Unless stated otherwise the multipliers discussed in this section have norms bounded by $p^*-1$ on $L^p(\Rd)$ for $1<p<\infty$, as results from the preceding discussion. We will focus on the symbols.

We note that $m(\xi)$ given by \eqref{mn:genc} is continuous in $\xi$, because so are $\Psi(\xi)$ and $\wt\Psi(\xi)$.  By \eqref{eq:dmf},  Plancherel theorem and \eqref{eq:psm1} for $p=2$ we also see that $|m(\xi)|\leq 1$.

Let $u>0$. We may consider $u\Psi$ and $u\wt \Psi$ instead of $\Psi$ and $\wt \Psi$
in \eqref{mn:genc}. If $\Ao=\At$, $\Re \Psi(A\xi)<0$ for $\xi\in \Rd$, and $u\to \infty$, then in the limit
we obtain the symbol 
\begin{align}\label{mn:gencd}
m(\xi)=\frac{\wt \Psi(A^T\xi)+\wt \Psi(-A^T\xi)}
{\Psi(A^T\xi)+ \Psi(-A^T\xi)}.
\end{align} 
Thus, the assumption $A=B$ rules out non-symmetric symbols. In fact, if $A\neq B$, then the corresponding L\'evy processes (see the proof of Theorem~\ref{th:bfm}) separate over time, and their parabolic martingales quickly decorrelate. We do not see a way to reproduce a nontrivial analogue of \eqref{mn:gencd} in this situation.
In this connection we also note that if $A=B=\I$ and $\Re\Psi(\xi)<0$, then \eqref{mn:gencd} is equivalent to 
\cite[(1.4)]{2011RBKBAB}. Furthermore, if $A\in \R^{d\times d}$ and $\det A\neq 0$, then multipliers corresponding to symbols $m(\xi)$ and $m(A^T\xi)$ have equal norms on $L^p(\R^d)$. In such a case \eqref{mn:gencd} is merely a trivial extension of \cite[(1.4)]{2011RBKBAB}.
If $\nu=0$, then \eqref{mn:gencd} yields, e.g., the symbols
\begin{equation}\label{eq:Bcs}
m(\xi)=\frac
{\int_{\sphere}\left( \xi, \theta \right)^2\varphi(\theta)\mu(d\theta)}
{\int_{\sphere}\left( \xi, \theta \right)^2\mu(d\theta)}, \quad \xi\in \Rd.
\end{equation}
Further discussion and examples related to \eqref{mn:gencd} may be found in \cite{2011RBKBAB}.
In particular \cite{2011RBKBAB} gives remarks on the integral form of the quadratic form (the second term) in \eqref{eq:l-c}, as opposed to the more usual matrix form, and yields the following symbols
\begin{align*}
  m(\xi)&=\frac{\ln \left(1+\xi_j^{-2}\right)}
  {
  \ln \left(1+\xi_1^{-2}\right)+\cdots+
  \ln \left(1+\xi_d^{-2}\right)}\, , \\
m(\xi)&=-2\xi_j\xi_k/|\xi|^2\ .
\end{align*}
Here $\xi\in \Rd\setminus\{0\}$, $j,k=1,\ldots,d$, and $j\neq k$.

To exhibit a non-symmetric symbol resulting from our construction, we let $n=d$, $\alpha \in (0,2)$ and $\Psi(\xi) = -|\xi|^{\alpha}$, so that   $\mu=0$, $\gamma=0$, $\nu (dz) = c_\alpha |z|^{d-\alpha}dz$, and $c_\alpha=\Gamma(\frac{d+\alpha}{2})2^\alpha\pi^{-d/2}/|\Gamma(-\frac{\alpha}{2})|$ in \eqref{eq:l-c} (see \cite{MR2569321}). These correspond to the isotropic $\alpha$-stable L\'evy process. If $\alpha\in (0,1)$ and $B= \I=-A$ in \eqref{mn:genc}, then by  \eqref{mn:gen3} and \eqref{eq:l-ctu},
\begin{eqnarray*}
&&m(\xi) = \frac{e^{-|2\xi|^{\alpha}} - e^{-2|\xi|^{\alpha}} }{-|2\xi|^{\alpha}+ 2|\xi|^{\alpha}}\int_{\Rd} \bigg( e^{i(\xi, z)} - 1\bigg)^2\phi(z)\nu (dz).
\end{eqnarray*}
Let $d=1$ and $\phi(z) = \sgn(z)$. We have
$( e^{i\xi z} - 1)^2=(e^{2i\xi z} - 1)-2( e^{i\xi z} - 1)$ and
\begin{eqnarray*}
&&\int_{\R}\frac{e^{i\xi z}-1}{|z|^{1+\alpha}}\phi(z)dz 
= 2i \int_0^{\infty} \frac{\sin{\xi z}}{|z|^{1+\alpha}}dz
= -2i\Gamma(-\alpha)\sin\frac {\pi\alpha}{2} \sgn(\xi)|\xi|^{\alpha} .
\end{eqnarray*} 
By this and the multiplication and reflection formulas for the gamma function,
\begin{align}\label{eq:wca}
\int_{\R} \bigg( e^{i\xi z} - 1\bigg)^2\phi(z)\nu (dz)
=-i\tan \frac{\pi\alpha}{2}\left[|2\xi|^\alpha-2|\xi|^\alpha\right].
\end{align}
Therefore,
\begin{eqnarray}\label{mn:ep}
m(\xi) = i\tan \frac{\pi \alpha}{2}
\sgn(\xi) (e^{-|2\xi|^{\alpha}} - e^{-2|\xi|^{\alpha}} ), \qquad\xi \in \R.
\end{eqnarray}
We may let $\alpha\to 1$ in \eqref{mn:ep}, and  use l'Hospital's rule to obtain
$$m(\xi) = \frac{4i\ln 2}{\pi}\; \xi\exp(-2|\xi|).$$
This agrees well with with \eqref{mn:gen4} and \eqref{mn:genca0}, see \eqref{eq:wca}.
By analytic continuation, \eqref{mn:ep} extends to $\alpha \in (1,2)$.

As seen in the proof of Theorem~\ref{th:bfm}, the drift $\gamma$ plays little role in our results, according with the conclusions of  \cite{2011RBKBAB}.

\section{Gaussian case}\label{sec:c}

For multipliers resulting from the linear transformations of the Brownian motion there is an alternative direct approach based on the classical It\^o calculus. The calculations are simpler and may shed some light on the procedures in Section~\ref{sec:j}. 
\begin{theorem}\label{th:m}
Let $d,n\in \N$ and $\Ao  ,\At \in \R^{d \times n}$. Let $K \in \C^{n \times n}$ satisfy
\begin{equation}\label{eq:o1}
|Kz|\leq |z| \quad \textrm{for } \;z\in \C^n\,.
\end{equation}  
For each $p \in (1,\infty)$, the Fourier multiplier $M$ with the symbol
\begin{equation}\label{br:mn}
m(\xi) = \bigg[e^{-|\Ao ^T\xi-\At ^T\xi|^2} 
- e^{-|\Ao ^T\xi|^2-|\At ^T\xi|^2}\bigg]\frac{(\Ao ^T\xi,K \At ^T\xi)}{(\Ao ^T\xi,\At ^T\xi)},
\end{equation}
is bounded in $L^p(\R^d)$. In fact, $\|M f\|_p\leq (p^*-1)\|f\|_p ~$ for $f\in \Lp(\Rd)$,
where we assume
$m(\xi) = e^{-|\Ao ^T\xi|^2-|\At ^T\xi|^2}(\Ao ^T\xi,K \At ^T\xi)$
if the denominator in \eqref{br:mn} is zero.
\end{theorem}
\begin{proof}
Let $(\mB_t)_{t\geq 0}$ be the Brownian motion in $\Rn$. 
Let $p_t$ denote the distribution of $\mB_t$. Thus, for $t>0$ we have $p_t(dw)=p_t(w)dw$, where  $p_t(w)= (2\pi t)^{-n/2}\exp(-|w|^2/(2t))$.
Let $f,g \in C_c^{\infty}(\R^d)$ and $x\in \Rd$. 
We consider the filtration
\begin{displaymath}
  \F_t=\sigma\{\mB_{s}\,;\; 0\leq s\leq t\}
\,,\quad t\geq 0\,,
\end{displaymath}
and the parabolic martingale $F_t= F_t(x;f,\Ao)$, where
\begin{eqnarray*}
F_t(x;f,\Ao) &=& \E [f(x+\Ao  \mB_1)|\F_t ]=\E [ f(x+\Ao  \mB_t + \Ao  (\mB_1-\mB_t)|\F_t]\\
&=& \int_{\R^d}f(x+\Ao  \mB_t +\Ao  z)p_{1-t}(dz).
\end{eqnarray*}
\noindent Note that $F_1 = f(x+\Ao  \mB_1)$ and $F_0 = \E f(x+\Ao  \mB_1).$ 
Let $\tilde{f}(z) = f(\Ao  z).$  We have $\nabla \tilde f(y) = \Ao  ^T \nabla f(\Ao  y)$.
For $0\leq t\leq 1,~w \in \R^d,$ we define
\begin{equation}
h(t,w) = \int_{\R^d} f(x+\Ao  w+\Ao  z)p_{1-t}(dz).
\end{equation}
We observe that $h$ is parabolic, i.e.
\begin{eqnarray}\label{eq4}
(\frac{\partial}{\partial t}+\frac{1}{2}\Delta_w)h(t,w) &=& \int_{\Rd} f(x+\Ao  w+\Ao  z) \frac{\partial}{\partial t} \big[ p_{1-t}(z)\big] dz\nonumber\\
&+& \frac12\int_\Rd\Delta_z [f(x+\Ao  w+\Ao  z)]p_{1-t}(z)dz= 0.
\end{eqnarray}
Here $\Delta_w=\sum_{i=1}^n \partial^2/\partial w_i^2$ is the Laplacian, and
the last inequality follows from integrating by parts and the 
 heat equation
\begin{displaymath}
\frac{\partial}{\partial s} p_s(z) = \frac{1}{2} \Delta_z p_s(z), \qquad s>0, ~z \in \Rn.
\end{displaymath}
Let $p^{\Ao}_t(dy)$ be the distribution of $\Ao  \mB_t$, i.e. $p^{\Ao}_t = \Ao  p_t$ (the pushforward measure). 
We have
\begin{displaymath}
\widehat{p^{\Ao}_t}(\xi) = \exp(-t|\Ao  ^T\xi|^2/2), \qquad \xi \in \R^d,
\end{displaymath} 
$h(t,w) = f * p^A_{1-t}(x+\Ao  w)$, and $h(1,w) = f(x+\Ao  w)$. Thus, $F_t(x;f,A) = h(t,W_t).$ By (\ref{eq4}) and It\^o formula for $h$ we obtain
\begin{equation}
F_t-F_0 = \int_0^t \Ao  ^T (\nabla f) * p^{\Ao}_{1-v} (x+\Ao  \mB_v)d\mB_v.
\end{equation}
For $t \in [0,u]$ we define
\begin{eqnarray*}
G_t&=&G_t(x;g,\At,K)=\int_0^t K \At ^T (\nabla g) * p^{\At}_{1-v} (x+\At \mB_v)d\mB_v,
\end{eqnarray*}
where $p_t^{B} = \At  p_t$.  The quadratic variations of these martingales are:
\begin{align}\label{eq:qvF}
[F,F]_t &= |F_0|^2 + \int_0^t  |\Ao  ^T (\nabla f) * p^{\Ao}_{1-v} (x+\Ao  \mB_v)|^2dv ,\\
\label{eq:qvG}
[G,G]_t &= \int_0^t |K \At ^T (\nabla g) * p^{\At}_{1-v} (x+\At \mB_v)|^2dv.
\end{align}
By Burkholder-Wang theory of differentially subordinated martingales \cite{MR1334160},
\begin{equation}\label{eq7}
\E |G_t(x;g,\At,K)|^p \leq (p^*-1)^p \E|F_t(x;g,\At)|^p.
\end{equation}
Therefore we have
\begin{eqnarray}\label{eq6}
\int_{\R^d}|F_1(x;f)|^p dx &=& \int_{\R^d} |f(x+\Ao  \mB_1)|^p dx= 
\int_{\R^d}\int_{\R^d} |f(x+\Ao  y)|^p p_1(dy)dx
\nonumber\\
&=& \int_{\R^d}\int_{\R^d} |f(x)|^p p_1(dy) dx =||f||_p^p.
\end{eqnarray}
A similar identitity holds for $g$ and $q=p/(p-1)$.
Therefore,
\begin{eqnarray}\label{eq8}
\int_{\R^d}\E|G_1(x;g,\At,K)|^p dx &\leq& (p^*-1)^p ||g||_p^p.
\end{eqnarray}
We define
$$\Lambda(f,g) = \int_{\R^d} \E [F,\overline{G}]_1dx.$$
By $(\ref{eq6}), (\ref{eq8})$ and H\"older inequality for the measure $P \otimes dx$, we have
\begin{eqnarray}\label{eq:oGm}
\Lambda(f,g)
&\leq& (p^*-1)||f||_q||g||_p.
\end{eqnarray}
By Plancherel theorem, 
\begin{eqnarray}
&&\Lambda(f,g)=\int_0^1 \int_{\R^d} (2 \pi)^{-d} \int_{\R^d}(\Ao  ^T\xi, K \At ^T \xi)
e^{-(1-t) |\Ao  ^T \xi|^2/2}\nonumber\\
&&\qquad\qquad \times e^{-(1-t) |\At ^T \xi|^2/2}e^{-i(\Ao  ^T \xi,y)}e^{i(\At ^T \xi,y)}p_t(y)\widehat{f}(\xi) \widehat{g}(-\xi)d\xi dy dt\nonumber\\
&=& \int_0^1\int_{\R^d} (2 \pi)^{-d}(\Ao ^T\xi, K\At ^T \xi)e^{-(1-t)(|\Ao ^T \xi|^2 + |\At ^T \xi|^2)/2}e^{-t |\At^T\xi -\Ao^T\xi|^2/2}\nonumber\\
&&\qquad\qquad\times\widehat{f}(\xi) \widehat{g}(-\xi)d\xi dt\nonumber\\
&=& \int_{\R^d}(2 \pi)^{-d} \widehat{f}(\xi) \widehat{g}(-\xi) (\Ao ^T\xi, K\At ^T \xi)e^{-(|\Ao ^T \xi|^2 +|\At ^T \xi|^2)/2}\nonumber\\
&&\qquad\qquad\times\int_0^1 e^{-t\left[|\At^T\xi -\Ao ^T\xi|^2-|\Ao ^T \xi|^2 - |\At ^T \xi|^2\right]/2}dtd\xi\label{eq:sii}\\
&=&\int_{\R^d}(2 \pi)^{-d} \widehat{f}(\xi) \widehat{g}(-\xi) (\Ao ^T\xi, K\At ^T \xi)e^{-(|\Ao ^T \xi|^2 + |\At ^T \xi|^2)/2}
\frac{e^{ (\Ao ^T \xi, \At ^T \xi)} -1}{(\Ao ^T\xi,\At ^T \xi)}d\xi\nonumber.
\end{eqnarray}
Here we used the identity
$ |\Ao ^T \xi|^2 + |\At ^T \xi|^2 - 2(\Ao ^T \xi, \At ^T \xi)  = 
|\At ^T\xi - \Ao ^T\xi|^2$ (if $(\Ao ^T \xi, \At ^T \xi)=0$, then the inner integral in \eqref{eq:sii} equals $1$).
The symbol $m$ obtains. The multiplier's norm bound follows from \eqref{eq:oGm}, as in the proof of Theorem~\ref{th:bfm}. 
\end{proof}

If $\Ao \xi = \At \xi \neq 0$ for all $\xi\neq 0$, and we multiply the matrices by $u \rightarrow \infty$, then 
$$m(\xi) = \frac{(\Ao ^T\xi,K\Ao ^T\xi)}{(\Ao ^T\xi,\Ao ^T\xi)},$$
obtains, and the corresponding multiplier has the same norm bound $p^*-1$ (see  remarks in Theorem~\ref{th:bfm}). Such symbols were discussed in some detail in \cite{MR2345912, 2011RBKBAB}.

\noindent
{\bf Acknowledgements.} {We thank Rodrigo Ba\~nuelos, Stanis{\l}aw Kwapie\'n, Remigijus Mikulevi\v{c}ius and Jacek Zienkiewicz for encoragement and discusions.}
Special thanks are due to Mateusz Kwa\'snicki for suggesting the unit time horizon for the parabolic martingales and confirming a part of our results through an independent calculation.


\end{document}